\documentclass[11pt,reqno,tbtags]{amsart}

\usepackage{amsthm}
\usepackage{amssymb}
\usepackage{amsfonts, mathrsfs}
\usepackage{amsmath}
\usepackage[numbers, sort&compress]{natbib}
\usepackage{graphicx}
\usepackage{vmargin, enumerate}
\usepackage{setspace}
\usepackage{paralist}
\usepackage[pdftex]{hyperref}

\newcommand{\N}{{\mathbb N}}

\newcommand{\E}[1]{{\mathbf E}\left[#1\right]}
\newcommand{\e}{{\mathbf E}}
\newcommand{\V}[1]{{\mathbf{Var}}\left\{#1\right\}}

\newcommand{\p}[1]{{\mathbf P}\left\{#1\right\}}

 \newcommand{\bag}{\begin{align}}
\newcommand{\bags}{\begin{align*}}
\newcommand{\eag}{\end{align*}}
\newcommand{\eags}{\end{align*}}

\newtheorem{thm}{Theorem}
\newtheorem{lem}[thm]{Lemma}

\newtheorem{cor}[thm]{Corollary}

\newcommand\cF{\mathcal F}

\newcommand\cT{{\mathcal T}}
\newcommand\cU{{\mathcal U}}

\newcommand{\refT}[1]{Theorem~\ref{#1}}

\newcommand{\refL}[1]{Lemma~\ref{#1}}

\newcommand{\pran}[1]{\left(#1\right)}

\hypersetup{
    bookmarks=true,         
    unicode=false,          
    pdftoolbar=true,        
    pdfmenubar=true,        
    pdffitwindow=true,      
    pdftitle={My title},    
    pdfauthor={Author},     
    pdfsubject={Subject},   
    pdfnewwindow=true,      
    pdfkeywords={keywords}, 
    colorlinks=true,       
    linkcolor=blue,          
    citecolor=blue,        
    filecolor=blue,      
    urlcolor=blue           
}

\newcommand\urladdrx[1]{{\urladdr{\def~{{\tiny$\sim$}}#1}}}

\begingroup
  \divide\count255 by 60
  \multiply\count255 by -60
  \advance\count255 by \time
  \ifnum \count255 < 10 \xdef\oclock{\the\count1:0\the\count255}
  \else\xdef\oclock{\the\count1:\the\count255}\fi
\endgroup

\renewcommand{\v}[1]{\mathit{var}\left\{#1\right\}}

\begin{document}

\title[Height and width of random trees]{Tail bounds for the height and width of a random tree with a given degree sequence}

\author{L. Addario-Berry}
\address{Department of Mathematics and Statistics, McGill University, 805 Sherbrooke Street West, 
		Montr\'eal, Qu\'ebec, H3A 2K6, Canada}
\email{louigi@math.mcgill.ca}
\date{September 22, 2010} 
\urladdrx{http://www.math.mcgill.ca/~louigi/}


\begin{abstract} 
Fix a sequence ${\bf c}=(c_1,\ldots,c_{n})$ of non-negative integers with sum $n-1$. We say a rooted tree $T$ has {\em child sequence} ${\bf c}$ if it is possible to order the nodes of $T$ as $v_1,\ldots,v_n$ so that for each $1 \le i \le n$, $v_i$ has exactly $c_i$ children. Let $\cT$ be a plane tree drawn uniformly at random from among all plane trees with child sequence ${\bf c}$. 
In this note we prove sub-Gaussian tail bounds on the height (greatest depth of any node) and width (greatest number of nodes at any single depth) of $\cT$. These bounds are optimal up to the constant in the exponent when ${\bf c}$ satisfies $\sum_{i=1}^n c_i^2=O(n)$; the latter can be viewed as a ``finite variance'' condition for the child sequence. 
\end{abstract}

\maketitle



\section{Introduction}\label{sec:intro} 

For a positive integer $n$, let ${\bf c}=(c_i)_{i=1}^n$ be a sequence of non-negative integers whose sum is $n-1$ (we call such a sequence a {\em child sequence}). In this paper we consider the {\em random plane tree} $\cT_{\bf c}$, chosen uniformly at random from the set of plane trees (rooted ordered trees) $T$ with $n$ nodes for which, for some ordering $v_1,\ldots,v_n$ of the nodes of $T$, node $v_i$ has $c_i$ children, for each $i \in [n]=\{1,\ldots,n\}$. The number of such trees is 
\begin{equation}\label{moon}
\frac{1}{n} \frac{n!}{\prod_{k=1}^n n_k!},
\end{equation}
where $n_k=n_k({\bf c})=\#\{i:c_i=k\}$ (see e.g., \cite{moon70counting}). 
For a given child sequence ${\bf c}$, we define the invariants 
\[
|{\bf c}| = \pran{\sum_{i=1}^n c_i^2}^{1/2} \quad \mbox{and} \quad 1_{\bf c} = \frac{n-2}{n-1-n_1(c)}.
\]
For a given tree $T$ and non-negative integer $i$, write $Z_i(T)$ for the number of nodes of $T$ at distance $i$ from the root. 
We then define 
\[
w(T) = \max \{ Z_i(T): i \in \N\}, \quad h(T) = \max \{ i: Z_i(T) \neq 0\},
\]
and call $w(T)$ and $h(T)$ the {\em width} and {\em height} of $T$, respectively. 
The main results of the paper are the following sub-Gaussian tail bounds on the width and height of $\cT_{\bf c}$, 
whose strength is controlled by the above invariants. 
\begin{thm}\label{thm:main}
For any $n \geq 1$ and all $m \geq 1$, 
\[
\p{w(\cT_{\bf c}) \geq m+2} \leq 3e^{-m^2/(1472 |{\bf c}|^2)} \quad \mbox{and} \quad 
\p{h(\cT_{\bf c}) \geq m} \leq 7e^{-m^2/(23552|{\bf c}|^21_{\bf c}^2)}.
\]
\end{thm}

\subsubsection*{\sc Remarks} 
\noindent \\$\star$~When $|{\bf c}|^2=O(n)$, this result is best possible up to the constants in the exponents. For the width, this follows from a connection, explained below, between the width and the fluctuations of random lattice paths. For the height, consider for example the special case where $n=2m+1$ and ${\bf c}$ consists of $m$ twos and $m+1$ zeros. Then $\cT_{\bf c}$ is a uniformly random binary plane tree, and in this case our bound (and the fact that it is tight) is a well-known result of Flajolet, Gao, Odlyzko and Richmond \citep[Theorem 1.3]{flajolet1993dhb}. 

\vspace{-0.4cm}
\noindent\\$\star$ A result related to Theorem~\ref{thm:main} appears in \cite{ab}. Fix a random variable $B$ with $\e{B}=1$ and $0\V{B} < \infty$. Then, for $n \ge 1$, let $T_n$ be a Galton--Watson tree with offspring distribution $B$, conditioned to have total progeny $n$. \citep[Theorems~1.1 and~1.2]{ab} then state that, for some $\epsilon>0$ not depending on $n$, $\p{w(T_n) > t} \le \exp(-\epsilon t^2/n)$, and if additionally $\V{B}>0$ then $\p{h(T_n) > t} \le \exp(-\epsilon t^2/n)$. The requirement that $\V{B}>0$ excludes the degenerate case where $\p{B=1}=1$. Note that if $B_1,\ldots,B_n$ are independent copies of $B$ then $\E{\sum_{i=1}^n B_i^2}= n\cdot \V{B}$, and so the finite variance condition would roughly correspond in our setting to a requirement that $|{\bf c}|^2 = O(n)$. Now temporarily write $C_1,\ldots,C_n$ for the numbers of children of the nodes of $T_n$ (note that $C_1,\ldots,C_n$ are exchangeable, but are not independent --- their sum is $n-1$ --- and are not distributed as $B$). We conjecture that in fact $n^{-1/2}(\sum_{i=1}^n C_i^2 - n\cdot\V{B})$ has Gaussian tails. A proof of this would show that the main results of \cite{ab} can be recovered from Theorem~\ref{thm:main}. 

\vspace{-0.4cm}
\noindent\\$\star$~In forthcoming work \cite{bm}, Broutin and Marckert use the tail bound for the height in Theorem~\ref{thm:main} as an ingredient in proving that, under suitable conditions on the child sequence, ${\bf c}$, the tree $\cT_{\bf c}$ converges in distribution to a Brownian continuum random tree after suitable rescaling. 

\vspace{-0.4cm}
\noindent\\$\star$~In \cite{abbg}, a bound very similar to the second bound of Theorem~\ref{thm:main} was required, for the height of a uniformly random labelled rooted tree of a fixed size. This bound was a key step in establishing the existence of a distributional Gromov--Hausdorff scaling limit for the sequence of rescaled components of a critical Erd\H{o}s--R\'enyi random graph $G_{n,p}$ when $p=p(n)$ is in the critical window $p-1/n=O(n^{-4/3})$. The results of this paper may thus be seen as a step towards establishing that the same scaling limit obtains for the sequence of components of a critical random graph with a given degree sequence \cite{hatami2011comp,joseph2011gr,riordan2011phase}. This is a line of enquiry that we shall pursue in a future paper. 

\vspace{-0.4cm}
\noindent \\$\star$~ The appearance of the term $1_{\bf c}$ in the bound on the height is necessary. 
For example, the sequence ${\bf c}=(1,1,\ldots,1,0)$ corresponds to a unique rooted plane tree, 
of height $n$. (For technical convenience, we exclude this unique, degenerate case from consideration 
for the remainder of the paper. Note that for any {\em other} child sequence ${\bf c}$, we have $|{\bf c}| \ge n$.)
More generally, given {\bf c}, define the {\em one-reduced} sequence ${\bf c}^*$, 
obtained by suppressing all entries of ${\bf c}$ which are equal to one. If 
${\bf c}$ has $k$ entries which are equal to one, then a tree with distribution 
$\cT_{\bf c}$ can then be generated from the tree $\cT_{\bf c^*}$ by repeatedly choosing 
a node $v$ uniformly at random, then subdividing the edge between $v$ and its parent (or, if 
$v$ happens to be the root, then adding a new node above $v$ and rerooting at this new node). 
Under this construction, each edge in $\cT_{\bf c^*}$ is subdivided $k/(n-k)$ times on average, 
and this is precisely the factor encoded 
by $1_{\bf c}$. 

The remainder of the note is devoted to proving Theorem~\ref{thm:main}. We first briefly describe a family of bijective correspondences between rooted plane trees and certain lattice paths; these correspondences allow us to prove bounds for the height and width by studying the fluctuations of a certain martingale. We accomplish this bounding by using a martingale concentration result of McDiarmid \cite{mcdiarmid98concentration}, which appears as Theorem~\ref{mcd}, below. This immediately yields the first bound in Theorem~\ref{thm:main}; the second requires a little further thought, and the use of a negative association result of Dubhashi~\cite{dubhashi1998balls}. Forthwith the details.

\subsection*{The Ulam--Harris tree, breadth-first search, lex-DFS and rev-DFS}

Below is a brief review of some basic connections between rooted plane trees and lattice paths. An 
excellent and detailed reference, with proofs, is \cite{legall2005rta}. 
The {\em Ulam--Harris} tree $\cU$ is the tree with root $\varnothing$ 
whose non-root nodes correspond to finite sequences of positive integers $v_1\ldots v_k$, 
with $v_1\ldots v_k$ having parent $v_1 \ldots v_{k-1}$ and children 
$\{v_1 \ldots v_{k} i~:~i \in \{1,2,\ldots\}\}$. For a node $v$ of $\cU$ we think of $vi$ as the 
$i$'th child of $v$. Any rooted plane tree $T$ in which all nodes have at most countably 
many children can be viewed as 
a subtree of $\cU$ by sending the root of $T$ to the root $\varnothing$ of $\cU$ 
and using the ordering of children in $T$ to recursively define an embedding of $T$ into $\cU$. 

Having viewed $T$ as a subtree of $\cU$, we now define three orderings on the nodes of $T$:
\begin{enumerate}
\item {\em breadth-first search (or BFS) order} lists the nodes of $T$ in increasing order of depth, and 
for nodes of the same depth, in lexicographic order (so, for example, node $2,3$ would appear before $3,1$ but after $1,7$); 
\item {\em lexicographic depth-first search (or lex-DFS) order} lists the nodes of $T$ in lexicographic order; 
\item {\em reverse lexicographic depth-first search (or rev-DFS)} is most easily described informally. 
Let $T^*$ be the mirror-image of $T$, and list the nodes of $T$ in the order they (their mirror images) appear in a lexicographic depth-first search of $T^*$. 
\end{enumerate}
The use of rev-DFS to bound heights of trees was introduced in \cite{ab}. 
Each of these orders have the property that when a node $v$ appears in the order, its parent in $T$ has already appeared. 
For such orders, we may define a {\em queue process}, as follows. Given the order $u_1,\ldots,u_n$ of the nodes of $T$, 
Let $Q_0=1$ and, for $1 \leq i \leq n$, let $Q_i=Q_{i-1}-1+c_{u_i}$, where $c_{u_i}$ is the number of children of $u_i$ in $T$. 
Then $Q_i$ is the number of nodes $u$ of the tree whose parent is among $u_1,\ldots,u_i$ but who are not themselves among $u_1,\ldots,u_i$. We will thus always have $Q_i > 0$ for $i < n$ and $Q_n=0$. We write $\{Q_i^b(T)\}_{0=1}^n$ for the queue process on the BFS order of $T$, and likewise define $\{Q_i^l(T)\}_{i=0}^n$ and $\{Q_i^r(T)\}_{0=1}^n$ for the lex-DFS and rev-DFS orders, respectively. 

Given the tree $T$, the preceding three processes are uniquely specified. Conversely, given any of the three sequences 
$\{Q_i^x(T)\}_{i=1}^n$, $x \in b,l,r$, the tree $T$ can be recovered. For each $x \in b,l,r$, this provides a bijection between 
rooted plane trees with $n$ nodes, on the one hand, and child sequences $(c_i)_{i=1}^n$ with $\sum_{1 \le i \le k} (c_i-1) \ge 0$ for all $1 \le k < n$. Call such sequences {\em tree sequences}. 

Given a sequence ${\bf c}=(c_i))_{i=1}^n$, set $S_0=0$ and $S_i=S_i({\bf c})=\sum_{j=1}^i (c_j-1)$ for $i \in [n]$. 
Also, given a permutation $\sigma:[n]\to[n]$, write $\sigma({\bf c})$ for the sequence $(c_{\sigma(i)})_{i=1}^n$. For a given sequence ${\bf c}=(c_1,\ldots,c_n)$ of non-negative integers with sum $n$, there is a unique cyclic permutation $\sigma=\sigma_{\bf c}:[n] \to [n]$ for which the sequence of partial sums $\sigma({\bf c})$ forms a tree sequence. 
(This fact yields a one-line proof of (\ref{moon}), above, by considering the number of permutations leaving ${\bf c}$ unchanged.) 
To be precise, $\sigma$ is the cyclic permutation sending $k$ to $n$, where $k$ is the least index at which the sequence $(S_i({\bf c}))_{i=0}^n$ achieves its minimum overall value.
Fix $x \in \{b,l,r\}$ and write $T^x({\bf c})$ for the tree $T$ corresponding to $\sigma({\bf c})$ under the $x$-bijection. It follows that letting $\tau$ be a uniformly random permutation of $[n]$, the tree $T^x({\bf \tau(c)})$ is a uniformly random tree with child sequence ${\bf c}$. Conversely, if $T$ is a uniformly random tree with child sequence ${\bf c}$, then $(Q_i^{x}(T))_{i=0}^n$ is distributed as $(S_i(\sigma_{\bf C}({\bf C})))_{i=0}^n$, where ${\bf C}=\tau({\bf c})$ and $\tau$ is a uniformly random permutation, independent of ${\bf c}$. 

\subsection*{Extremes in a sequence and its permutations}
In what follows, for positive integers $p,q$, we write $(p)_q = p \mod q$ if $q\nmid p$ and $(p)_q = q$ if $q \mid p$. 
For this section, fix a child sequence ${\bf c}=(c_1,\ldots,c_n)$ and $x \in \{b,l,r\}$, and let $\sigma=\sigma_{\bf c}$. 
Note that $(S_i(\bf c))_{i=0}^n$ has $S_i > 0$ for all $i < n$ precisely if ${\bf c}$ is a tree sequence. 
\begin{lem}\label{halfsplit}
If $\max_{0 \leq i \leq n} S_i(\sigma({\bf c}))= m$ then $\max_{0 \leq i \leq n} |S_i({\bf c})| \geq m/2$. 
\end{lem}
\begin{proof}
For $x \in [0,2n]$, write $x_n=x$ if $x \leq n$ and $x_n=x-n$ if $x > n$. 
Since $\sigma$ is a cyclic shift, there is $j$ so that $\sigma(i)=(i+j)_n$ for $j \in [n]$. 
Let $k$ be the index at which $S_k(\sigma({\bf c}))=m$, so that 
\[
m= \sum_{i=1}^k (c_{\sigma(i)}-1) = \sum_{i=1}^k (c_{(j+i)_n} - 1). 
\]
If $j+i \leq n$ then $S_{j+i} = S_j+m$ so either $S_j \leq -m/2$ or $S_{j+i} \geq m/2$. 
If $j+i > n$ then $m=(S_n-S_j)+S_{j+i-n}$ so either $S_j \leq S_n-m/2 = -1-m/2$ or $S_{j+i-n} \geq m/2$. 
\end{proof}
Let $\sigma^*$ be the cyclic permutation sending $1$ to $1+\lfloor n/2 \rfloor$. We then immediately have the following corollary.
\begin{cor} \label{quarter}
If $\max_{0 \leq i \leq n} S_i(\sigma({\bf c}))= m$ then either 
\[
\max_{0 \leq i \leq \lfloor n/2 \rfloor} |S_i| \geq m/4
\quad
\mbox{or}
\quad
\max_{0 \leq i \leq \lceil n/2 \rceil} |S_{\sigma^*(i)}| \geq m/4.
\]
\end{cor}
\begin{proof}
By \refL{halfsplit}, we have $\max_{0 \leq i \leq n} |S_i({\bf c})| \geq m/2$, 
so one of these two alternatives must occur. 
\end{proof}

\subsection*{Martingales for the queue processes}
We will use a martingale inequality that can be found in \cite{mcdiarmid98concentration}. 
Let $\{X_i\}_{i=0}^n$ be a bounded martingale adapted to a filtration $\{\cF_i\}_{i=0}^n$. 
Let $V=\sum_{i=0}^{n=1} \v{X_{i+1}|\cF_i}$, where 
\[
\v{X_{i+1}|\cF_i} := \E{(X_{i+1} - X_i)^2|\cF_i} = \E{X_{i+1}^2|\cF_i}-X_i^2 
\]
is the predictable quadratic variation of $X_{i+1}$. 
Define 
\[
v=\mathop{\mathrm{ess}}\sup V,
\quad
\mbox{and}
\quad
b = \max_{0 \leq i \leq n-1} \mathop\mathrm{ess}\sup (X_{i+1}-X_i| \cF_i).
\]
Then we have the following bound.
\begin{thm}[\cite{mcdiarmid98concentration}, Theorem 3.15]\label{mcd}
For any $t \geq 0$, 
\[
\p{\max_{0 \leq i \leq n} X_i \geq t} \leq \exp\pran{-\frac{t^2}{2v(1+bt/(3v))}}.
\]
\end{thm}
In \cite{mcdiarmid98concentration}, this result is stated for $\p{X_n \geq t}$ rather than 
for the supremum of the $X_i$ as above. However, as noted by McDiarmid, 
the proof is based on bounding $\E{e^{hX_n}}$ for suitably chosen $h>0$. 
Since $\{e^{hX_i},0 \leq i \leq n\}$ is a submartingale, 
the version for the supremum in fact holds by a simple application of Doob's inequality. 

Now, fix a child sequence ${\bf c}=(c_1,\ldots,c_n)$ and $x \in \{b,l,r\}$. 
Let $\tau:[n]\to[n]$ be a uniformly random permutation, and 
write ${\bf C}=(C_1,\ldots,C_n)=(\tau({\bf c}))$. 

For $0 \leq k \leq n-1$ let $n_k^0 = \#\{i: C_i=k\} = n_k({\bf c})$. For $i > 0$ and $0 \leq k \leq n-1$, define 
\[
n_k^i=n_k^i({\bf C}) = 	\begin{cases}
						n_k^{i-1}		& \mbox{if}~C_i \neq k, \\
						n_k^{i-1}-1 	& \mbox{if}~C_i = k.
					\end{cases}
\]
Then for all $0 \leq i \leq n$, $\sum_{k=0}^{n-1} n_k^i = n-i$. 
Also, for each $1 \leq i \leq n$, there is a single $k$ with $n_k^i\neq n_k^{i-1}$, and furthermore, for this $k$, 
$S_i({\bf C}) = S_{i-1}({\bf C}) + k - 1$. 
Thus, for all $0 \leq i \leq n$, 
\[
\sum_{k=0}^{n-1} k n_k^{i} + S_i = \sum_{k=0}^{n-1} k n_k^{0} - i = n-1-i.
\] 
Writing $\cF_i$ for the sigma-field generated by $S_0,\ldots,S_i$, 
we then have 
\[
\E{S_{i+1}| \cF_i} = S_i + \sum_{j=0}^{n-1} (j-1) \frac{n_j^i}{n-i} 
= S_i - \frac{S_i+1}{n-i}, 
\]
and
\begin{align*}
\E{S_{i+1}^2|\cF_i}
& = \sum_{k=0}^{n-1} (S_i+(k-1))^2 \frac{n_k^i}{n-i}\\ 
& = S_i^2 - \frac{2S_i(S_i+1)}{n-i} + \frac{\sum_{k=0}^{n-1} (k-1)^2 n_k^i}{n-i}. \\ 
\end{align*} 
At this point it would be natural to turn to the study of the martingale whose value at time $i$ is $S_i+\sum_{j=0}^{i-1} (S_j+1)/(n-j)$, or in other words to subtract off the predictable part. However, this would require us to separately bound the sums of the $(S_j+1)/(n-j)$, and a more direct route is to simply bound these summands directly. 
From the preceding equations, for $i < n$ we have 
\[
\E{\frac{S_{i+1}+1}{n-(i+1)}|\cF_i} = \frac{S_i + 1}{n-(i+1)} - \frac{S_i+1}{(n-i)(n-(i+1))} = \frac{S_i+1}{n-i}.
\]
Here we take $0/0=1$ by convention to deal with the term $i=n-1$. Thus, $M_i = (S_i+1)/(n-i)$ is an $\cF_i$-martingale. 
Since $S_{i+1} \geq S_i-1$ for each $i < n$, for $i < \lfloor n/2 \rfloor$ we have 
\[
M_{i+1}=\frac{S_{i+1}+1}{n-(i+1)}  \geq \frac{S_i+1}{n-(i+1)} - \frac{2}{n} 
					 = \frac{S_i+1}{n-i} - \frac{S_i+1}{(n-i)(n-(i+1))} - \frac{2}{n} 
					  \geq \frac{S_i+1}{n-i} - \frac{4}{n} = M_i-\frac{4}{n},
\]
which we will use below when applying Theorem~\ref{mcd}. 
We also have 
\begin{align*}
\v{M_{i+1}|\cF_i} & = \E{M_{i+1}^2|\cF_i}-M_i^2 \\
					& = \frac{1}{n-(i+1)^2} (\E{S_{i+1}^2|\cF_i} + 2\E{S_{i+1}|\cF_i} + 1) - \pran{\frac{S_i+1}{n-i}}^2\\
					& = \frac{1}{(n-(i+1))^2}
					\pran{3 + \frac{\sum_{k=0}^{n-1} (k-1)^2 n_k^i}{n-i}} - \pran{\frac{S_i+1}{(n-i)(n-(i+1))}}^2\\
					& \leq \frac{1}{(n-(i+1))^2} \pran{3 + \frac{\sum_{i=1}^n c_i^2}{n-i}}.
\end{align*}
Writing $a=\sum_{i=1}^n c_i^2/n$, for $i < \lfloor n/2 \rfloor$ we obtain the bound
\[
\v{M_{i+1}|\cF_i} \leq \frac{4(3+2a)}{n^2}, 
\]
and so 
\[
\sum_{i=1}^{\lfloor n/2\rfloor}\v{M_{i}|\cF_{i-1}} \leq \frac{4(3+2a)}{n}.
\]
It follows by applying \refT{mcd} to $\{-M_i\}_{i=0}^{\lfloor n/2 \rfloor}$ that for all $t \geq 0$, 
\begin{align}
\p{\min_{0 \leq i \leq \lfloor n/2 \rfloor} S_i \leq -(t+1)} & \leq \p{\min_{0 \leq i \leq \lfloor n/2 \rfloor} \frac{S_i+1}{n-i} \leq -\frac{t}{n}}\nonumber\\
& \leq \exp\pran{-\frac{t^2}{n\cdot8(3+2a)(1+t/(3(3+2a)n))}} \nonumber\\
& = \exp\pran{-\frac{t^2}{8(3+2a)n + 8t/3}}.\label{halfway}
\end{align}

Recall that $\sigma_{\bf c}$ is the unique cyclic permutation $\sigma$ which makes $\sigma({\bf c})$ a tree sequence.
We are now prepared for our principal bound on the fluctuations of $\{S_i(\sigma({\bf c})),0 \leq i \leq n\}$. 
\begin{thm}\label{rotategauss} 
For any non-negative integer $m$, 
\[
\p{\max_{0 \leq i \leq n} S_i(\sigma_{\bf c}({\bf c})) \geq m+2} \leq 3\exp\pran{-\frac{m^2}{368 |{\bf c}|^2}}. 
\] 
\end{thm}
\begin{proof}
First, since $\sigma_{\bf c}({\bf c}) = \sigma_{\bf C}({\bf C})$, it suffices to bound 
$\p{\max_{0 \leq i \leq n} S_i(\sigma({\bf C})) \geq m+2}$, which is what we shall do. 
Also, for $m \ge n-3$ the event under consideration can never occur, so we may and shall assume $m < n-3$. 
Finally, for this proof, by $S_i$ we mean $S_i({\bf C})$ unless an argument is provided. 

First note that if $\max_{0 \leq i \leq n} S_i(\sigma({\bf C})) = m+2$, then 
\begin{equation}\label{mplusone}
\max_{0 \leq i \leq n} S_i-\min_{0 \leq i \leq n} S_i = m+3. 
\end{equation}
(In fact, the same must hold for any cyclic permutation of ${\bf C}$.) 
This will imply that at some point, $\{S_i,0 \leq i \leq n\}$ drops in value significantly. 
Let $m_0 = \max_{0 \leq i \leq \lfloor n/2 \rfloor} S_i$, and consider the following two events. 
\begin{itemize}
\item[(a)] $\min_{0 \leq i \leq \lfloor n/2 \rfloor} S_i \leq -(m+3)/3$
\item[(b)] $S_{\lfloor n/2 \rfloor} \leq m_0 -(m+3)/3$. 
\end{itemize}
If (a) does not occur then 
$\{S_0,S_1,\ldots,S_{\lfloor n/2 \rfloor}\} \subset (-(m+3)/3,m_0)$. 
Thus, if neither (a) nor (b) occur then for (\ref{mplusone}) to hold one of the following must take place.
\begin{itemize}
\item[(c)] $m_0 > 2(m+1)/3$,
\item[(d)] $\max_{\lfloor n/2 \rfloor < i \leq n} S_i > 2(m+3)/3$,
\item[(e)] $\min_{\lfloor n/2 \rfloor < i \leq n} S_i < m_0-(m+3)$. 
\end{itemize}
If (b) does not occur but (c) occurs then $S_n-S_{\lfloor n/2 \rfloor} < -(m+3)/3$. If (d) occurs 
then since $S_n=-1$, $S_n - \max_{\lfloor n/2 \rfloor < i \leq n} S_i < -2(m+3)/3$. 

Now note that if either (a) or (b) occurs then 
\[
\min_{0 \leq i \leq \lfloor n/2 \rfloor} (S_{\lfloor n/2 \rfloor} - S_{\lfloor n/2 \rfloor - i}) \leq -(m+3)/3,
\]
and if (b) does not occur but one of (c),(d) does then $S_n - \max_{\lfloor n/2 \rfloor < i \leq n} S_i < -(m+3)/3$, and so 
\[
\min_{0 \leq i \leq \lceil n/2 \rceil} (S_n-S_{n-i}) < -(m+3)/3. 
\]
Finally, if (b) does not occur but (e) occurs then since $S_{\lfloor n/2\rfloor} > m_0-(m+3)/3$, 
we have 
\[
\min_{\lfloor n/2 \rfloor \leq i \leq n} (S_i-S_{\lfloor n/2 \rfloor}) < -2(m+3)/3.
\]
Since $(S_{\lfloor n/2 \rfloor} - S_{\lfloor n/2 \rfloor - i},0 \leq i \leq \lfloor n/2 \rfloor)$ 
has the same distribution as $(S_i,0 \leq i \leq \lfloor n/2\rfloor)$, and 
$(S_i-S_{\lfloor n/2\rfloor}, \lfloor n/2 \rfloor \leq i \leq n)$ and 
$(S_n-S_{n-i},0 \leq i \leq \lceil n/2 \rceil)$ both have the same distribution as $(S_i,0 \leq i \leq \lceil n/2 \rceil)$, 
it follows that 
\begin{align*}
\p{\max_{0 \leq i \leq n} S_i(\sigma_{\bf C}({\bf C})) \geq m+2} 
	& \leq 3\p{\min_{0 \leq i \leq \lfloor n/2 \rfloor} S_i \leq -(m/3+1)}\\
	& \leq 3\exp\pran{-\frac{m^2}{72(3+2a) n+8m}},
\end{align*}
the latter bound holding by (\ref{halfway}). Since $m < n-3$, we have $72\cdot 3n+8 m < 224 n \le 224 |c|^2$. Also, $72\cdot2an=144 |c|^2$, and the result follows. 
\end{proof}

\subsection*{Bounding the width and the height} 

The bounds of Theorem~\ref{thm:main} follow straightforwardly from Theorem~\ref{rotategauss}. Let $\mathcal{T}_{\bf c}$ be a uniformly random tree with child sequence ${\bf c}$. As noted earlier, $(Q^b_i(\mathcal{T}_{\bf c}))_{i=0}^n$ is distributed as $(S_i(\sigma_{\bf C}({\bf C})))_{i=0}^n$, where ${\bf C}=\tau({\bf c})$ and $\tau$ is a uniformly random permutation, independent of ${\bf c}$. 
Furthermore, when the breadth-first exploration has just finished exploring all the nodes at depth $k$, the queue length is precisely the number of nodes at depth $k-1$. It follows by Theorem~\ref{rotategauss} that 
\[
\p{w(\mathcal{T}_c) \ge m+2} \le \p{\max_{0 \le i \le n} Q^b_i(\mathcal{T}_{\bf c}) \ge m+2} \le 3\exp\pran{-\frac{m^2}{368|{\bf c}|^2}},
\]
proving the bound for the width. (Also, if at some point the queue length is at least $m$ then $w(\cT_{\bf c}) \ge m/2$, from which the optimality of the with bound when $|{\bf c}|^2=O(n)$ follows straightforwardly.)

In bounding the height, we assume that $m \ge 6\sqrt{n}$, or else the bound follows trivially since $|{\bf c}|^2 \ge n$. 
First suppose that ${\bf c}$ is one-reduced (so has no entries equal to one). For any node $u \in T$, let $\lambda(u)$ (resp.~$\rho(u)$) be the index of $u$ when the nodes of $T$ are listed in lex-DFS order (resp.~rev-DFS order). Since ${\bf c}$ is one-reduced, each ancestor of $u$ in $T$ has at least one child that is not an ancestor of $u$, and so 
either $Q^l_{\lambda(u)}(T)$ or $Q^r_{\rho(u)}(T)$ is at least $|u|/2$. It follows that when ${\bf c}$ is one-reduced, 
\begin{align*}
\p{h(\mathcal{T}_c) \ge m+4 } & \le \p{\max_{0 \le i \le n} Q^l_i(\mathcal{T}_{\bf c}) \ge \lceil m/2\rceil+2} + \p{\max_{0 \le i \le n} Q^r_i(\mathcal{T}_{\bf c}) \ge \lceil m/2\rceil+2} \\
		& \le 6\exp\pran{-\frac{m^2}{1472|{\bf c}|^2}},
\end{align*}
proving the bound in this case. 


More generally, write ${\bf c^*}$ for the one-reduced version of ${\bf c}$, obtained from ${\bf c}$ by removing all entries that are equal to one, and let $n^*$ be the length (number of elements) of ${\bf c^*}$. Also, write $\mathcal{T}^*$ for the tree obtained from $\mathcal{T}_{\bf c}$ by replacing each maximal path whose internal nodes all have exactly one child, by a single edge. List the edges of $\mathcal{T}^*$ according to some fixed rule as $(e_1,\ldots,e_{n^*-1})$. Note that we always have $n^* \le n-1$. Each edge $e_i$ corresponds to some path in $\mathcal{T}_c$, and we write $s_i$ for the number of internal nodes of this path (i.e. the total number of nodes, minus two). Then $\mathcal{T}^*$ is distributed as a uniformly random tree with child sequence ${\bf c^*}$, and, independently of $\mathcal{T}^*$, $(s_1,\ldots,s_{n^*-1})$ is a uniformly random element of the set of vectors of non-negative integers of length $n^*-1$ with sum $n-n^*-1$. 
From Theorem~\ref{rotategauss} we thus have 
\begin{equation}\label{eq:reduced_bd}
\p{h(\mathcal{T}^*) \ge m+4} \le 6\exp\pran{-\frac{m^2}{1472|{\bf c^*}|^2}} \le 6\exp\pran{-\frac{m^2}{1472|{\bf c}|^2}}. 
\end{equation}
If $n^* \ge n-\sqrt{n}$ then $h(\mathcal{T}_{\bf c}) \le h(\mathcal{T}^*)+\sqrt{n}$, and in this case the required bound follows (recall that we have shown we may assume $m \ge 6\sqrt{n}$). In what follows we thus assume $n^* < n-\sqrt{n}$. 

By Proposition 5 of \cite{dubhashi1998balls}, the entries of $(s_1,\ldots,s_{n^*-1})$ are negatively correlated and thus standard Chernoff bounds apply to any restricted sum of elements of $(s_1,\ldots,s_{n^*-1})$. In particular, for any node $v$ of $\mathcal{T}^*$, 
\[
J_v = \{i : e_i \text{ is an edge of the path from $v$ to the root of $\mathcal{T}^*$}\}. 
\]
We always have $J_v \le h(\mathcal{T^*})-1$, and thus by a Chernoff bound (e.g., \cite{mcdiarmid98concentration}, Theorem 2.2), 
\[
\p{\left.S_{J_v} \ge (1+x)(m+2)\frac{n-n^*-1}{n^*-1} \right| h(\mathcal{T^*}) \le m+3} \le \exp\pran{-2x^2(m+2)^2 \frac{n-n^*-1}{n^*-1}}. 
\]
To get a clean final bound, we choose $x$ so that $(1+x)(m+2)=2m$. It then follows by a union bound that 
\begin{align*}
	& \p{\exists v \in \mathcal{T}^*: \left.S_{J_v} \ge 2m\frac{n-n^*-1}{n^*-1} \right| h(\mathcal{T^*}) \le m+3} \\
\le 	& \exp\pran{\log(n^*-1) - 2(m-2)^2 \frac{n-n^*-1}{n^*-1}} \\
\le 	& \exp\pran{- (m-2)^2 \frac{n-n^*-1}{n^*-1}} \\
\le	 & \exp\pran{- \frac{m^2 (n-n^*-1)}{9|{\bf c}|^2}} 
\end{align*}
the second inequality holding since $(m-2)^2 \ge n^*-1$ and $n-n^*-1 \ge \sqrt{n}-1 \ge \log(n^*-1)$, and the third holding since 
$|{\bf c}|^2 \ge n > n^*-1$ and $(m-2) \ge m/3$. Since $m+4 \le 2m = 2m(n^*-1)/(n^*-1)$, it then follows from (\ref{eq:reduced_bd}) that 
\[
\p{h(\mathcal{T}_{\bf c}) \ge 4m\frac{n-2}{n^*-1}} \le 7 \exp\pran{-\frac{m^2}{1472|{\bf c}|^2}}.
\]
But $(n-2)/(n^*-1) = 1_{\bf c}$, and the result follows. 
              
   
%

\bibliographystyle{plainnat}
\bibliography{hds}

\begin{thebibliography}{11}
\providecommand{\natexlab}[1]{#1}
\providecommand{\url}[1]{\texttt{#1}}
\expandafter\ifx\csname urlstyle\endcsname\relax
  \providecommand{\doi}[1]{doi: #1}\else
  \providecommand{\doi}{doi: \begingroup \urlstyle{rm}\Url}\fi

\bibitem[Addario-Berry et~al.(2010)Addario-Berry, Devroye, and Janson]{ab}
L.~Addario-Berry, L.~Devroye, and S.~Janson.
\newblock Sub-{G}aussian tail bounds for the width and height of conditioned
  {G}alton--{W}atson trees.
\newblock arXiv:1011.4121v1 [math.PR], November 2010.

\bibitem[Addario-Berry et~al.(2011+ (in press))Addario-Berry, Broutin, and
  Goldschmidt]{abbg}
L.~Addario-Berry, N.~Broutin, and C.~Goldschmidt.
\newblock The continuum limit of critical random graphs.
\newblock \emph{Probability Theory and Related Fields}, 2011+ (in press).

\bibitem[Broutin and Marckert()]{bm}
N.~Broutin and J.F. Marckert.
\newblock Asymptotics for trees with a prescribed degree sequence, and
  applications.
\newblock in preparation.

\bibitem[Dubhashi and Ranjan(1998)]{dubhashi1998balls}
D.~Dubhashi and D.~Ranjan.
\newblock Balls and bins: A study in negative dependence.
\newblock \emph{Random Structures and Algorithms}, 13\penalty0 (2):\penalty0
  99--124, 1998.

\bibitem[Flajolet et~al.(1993)Flajolet, Gao, Odlyzko, and
  Richmond]{flajolet1993dhb}
P.~Flajolet, Z.~Gao, A.~Odlyzko, and B.~Richmond.
\newblock {The distribution of heights of binary trees and other simple trees}.
\newblock \emph{Combinatorics, Probability and Computing}, 2\penalty0
  (2):\penalty0 145--156, 1993.

\bibitem[Hatami and Molloy(2011+ (in press))]{hatami2011comp}
H.~Hatami and M.~Molloy.
\newblock The scaling window for a random graph with a given degree sequence.
\newblock \emph{Random Structures and Algorithms}, 2011+ (in press).

\bibitem[Joseph(2011+)]{joseph2011gr}
A.~Joseph.
\newblock The component sizes of a critical random graph with a given degree
  sequence.
\newblock arXiv:1012.2352v2 [math.PR], 2011+.

\bibitem[Le~Gall(2005)]{legall2005rta}
J.F. Le~Gall.
\newblock {Random trees and applications}.
\newblock \emph{Probability Surveys}, 2:\penalty0 245--311, 2005.

\bibitem[McDiarmid(1998)]{mcdiarmid98concentration}
C.~McDiarmid.
\newblock Concentration.
\newblock In M.~Habib, C.~McDiarmid, J.~Ramirez-Alfonsin, and B.~Reed, editors,
  \emph{Probabilistic Methods for Algorithmic Discrete Mathematics}, pages
  195--248, New York, 1998. Springer Verlag.

\bibitem[Moon(1970)]{moon70counting}
J.W. Moon.
\newblock \emph{Counting labelled trees}.
\newblock Number~1. Canadian Mathematical Monographs, 1970.

\bibitem[Riordan(2011+)]{riordan2011phase}
O.~Riordan.
\newblock The phase transition in the configuration model.
\newblock arXiv:1104.0613v1 [math.PR], 2011+.

\end{thebibliography}
                 

\end{document}